\documentclass[12pt]{article}

\usepackage[margin=1in]{geometry}  
\usepackage{graphicx}              
\usepackage{amsmath}               
\usepackage{amsfonts}              
\usepackage{amsthm}                

\newtheorem{thm}{Theorem}

\newtheorem{conj}[thm]{Conjecture}

\newcommand{\R}{\mathbb{R}}

\begin{document}

\title{A one-dimensional symmetry result\\
for solutions of an integral equation\\
of convolution type\thanks{This work has been carried out in
the framework of the Labex Archim\`ede (ANR-11-LABX-0033),
of the A*MIDEX project (ANR-11-IDEX-0001-02)
and of the WIAS ERC-1 project.
The research leading to these results has received funding from the 
European Research Council 
Grant n.~321186 - ReaDi - 
``Reaction-Diffusion Equations, Propagation and Modelling''
and n.~277749 - EPSILON -
``Elliptic Pde's and Symmetry of 
Interfaces and Layers for Odd Nonlinearities'',
and the PRIN Grant n.~201274FYK7 ``Critical Point Theory
and Perturbative Methods for Nonlinear Differential Equations''.
Part of this 
work was carried out during a visit by F. Hamel to the 
Weierstra{\ss} Institute, whose 
hospitality is thankfully acknowledged.}}

\author{Fran{\c{c}}ois Hamel\\
{\small Universit\'e d'Aix-Marseille}\\
{\small Institut de Math\'ematiques de Marseille}\\
{\small 39 rue Fr\'ed\'eric Joliot-Curie}\\
{\small 13453 Marseille Cedex 13, France}\\
{\small {\tt{francois.hamel@univ-amu.fr}} } \\
\\
and\\
\\
Enrico Valdinoci\\
{\small Weierstra{\ss} Institute}\\
{\small Mohrenstra{\ss}e 39}\\
{\small 10117 Berlin, Germany}\\
{\small and}\\
{\small Universit\`a di Milano}\\
{\small Dipartimento di Matematica Federigo Enriques}\\ 
{\small Via Cesare Saldini 50}\\
{\small 20133 Milano, Italia}\\
{\small {\tt{enrico@math.utexas.edu}} }
\\ } 

\maketitle

\newpage

\begin{abstract}
We consider an integral equation in the plane,
in which the leading operator is of convolution type,
and we prove that monotone (or stable) solutions
are necessarily one-dimensional.
\end{abstract}

\bigskip


\bigskip

\section*{Introduction}

In this paper, we consider solutions of an integral
equation driven by the following nonlocal, linear operator
of convolution type:
\begin{equation}\label{OP} {\mathcal{L}} u(x):= \int_{\R^n} 
\big( u(x)-u(y)\big)\,k(x-y)\,dy.\end{equation}
Here we suppose\footnote{For the sake of completeness,
we point out that assumptions more general than~\eqref{ASS-kernel}
may be taken into account with the same methods as the ones
used in this paper. For instance, one could follow assumptions~(H1)--(H4)
in~\cite{Coville} with~$g\ge\alpha$ for some~$\alpha>0$.
We focus on the simpler case of assumption~\eqref{ASS-kernel}
for simplicity.}
that~$k$ is an even, measurable kernel with normalization
$$ \int_{\R^n} k(\zeta)\,d\zeta=1$$
and
such that
\begin{equation}\label{ASS-kernel}
m_o\chi_{B_{r_o}}(\zeta)\le k(\zeta)\le
M_o\chi_{B_{R_o}}(\zeta)
\end{equation}
for any~$\zeta\in\R^n$, for some fixed~$M_0\geq m_0>0$ and~$R_0\geq r_0>0$.

We consider here solutions~$u$ of the semilinear 
equation
\begin{equation}\label{EQ}
{\mathcal{L}} u(x) = f\big(u(x)\big).
\end{equation}
In the past few years, there has been an intense activity
in this type of equations, both for its mathematical interest
and for its relation with biological models, see, among the
others~\cite{25, 26, 32, 34}. In this case, the solution~$u$
is thought as the density of a biological species
and the nonlinearity~$f$ is often a logistic map,
which prescribes the birth and death rate of the population.
In this framework,
the nonlocal diffusion modeled by~${\mathcal{L}}$
is motivated by the long-range interactions between
the individuals of the species.
\medskip

The goal of this paper is to study the symmetry properties
of solutions of~\eqref{EQ} in the light of a famous conjecture
of De Giorgi arising in elliptic partial differential
equations, see~\cite{EDG}. The original problem
consisted in the following question:

\begin{conj}\label{C:DG}
Let~$u$ be a bounded solution of
$$ -\Delta u=u-u^3$$
in the whole of~$\R^n$, with
\begin{equation*}
{\mbox{$\partial_{x_n} u(x)>0$
for any $x\in\R^n$.}}\end{equation*}
Then, $u$ is necessarily one-dimensional, i.e. there exist~$u_\star:\R\to\R$
and~$\omega\in \R^n$ such that~$u(x)=u_\star(\omega\cdot x)$,
for any~$x\in\R^n$, at least when~$n\le8$.
\end{conj}

The literature has presented several variations of Conjecture~\ref{C:DG}:
in particular, a weak form of it has been investigated
when the additional assumption
\begin{equation}\label{LIM}
\lim_{x_n\to\pm\infty} u(x_1,\dots,x_n)=\pm1
\end{equation}
is added to the hypotheses.

When the limit in~\eqref{LIM} is uniform in the
variables~$(x_1,\dots,x_{n-1})\in\R^{n-1}$,
the version of Conjecture~\ref{C:DG} obtained in this
way is due to Gibbons and is related to problems in
cosmology.

In spite of the intense activity of the problem,
Conjecture~\ref{C:DG} is still open in its generality.
Up to now, Conjecture~\ref{C:DG} is known to have a positive
answer in dimension~$2$ and~$3$ (see~\cite{GG, AC} and
also~\cite{BCN, AAC})
and a negative answer in dimension~$9$ and higher (see~\cite{DKW}).

Also, the weak form of
Conjecture~\ref{C:DG} under the limit assumption in~\eqref{LIM}
was proved (up to the optimal dimension~$8$) in~\cite{S},
and the version of Conjecture~\ref{C:DG}
under a uniform limit assumption in~\eqref{LIM}
holds true in any dimension
(see~\cite{F, BBG, BHM}).

Since it is almost impossible to keep track in this short introduction
of all the research developed on this important topic,
we refer to~\cite{FV} for further details and motivations.
\medskip

Goal of this paper is to investigate
whether results in the spirit of 
Conjecture~\ref{C:DG} hold true when the Laplace operator
is replaced by the nonlocal, integral operator in~\eqref{OP}.
We remark that symmetry results in nonlocal settings
have been obtained in~\cite{CSM, SV, LV, CS, CC1, CC2},
but all these works dealt with fractional operators with
a regularizing effect. Namely, the integral kernel
considered there is not integrable, therefore the solutions
of the associated equation enjoy additional regularity
and rigidity properties. Also, some of the problems
considered in the previous works rely on an extension
property of the operator that bring the problem
into a local (though higher dimensional and either singular
or degenerate) problem.

In this sense,
as far as we know, this paper is the first one to take into
account integrable kernels, for which the above
regularization techniques do not hold and for which
equivalent local problems are not available.
\medskip

In this note, we prove the following
one-dimensional result in dimension~$2$:

\begin{thm}\label{DG-2}
Let~$u$ be a solution of~\eqref{EQ}
in the whole of~$\R^2$, with~$\|u\|_{C^1(\R^2)}<+\infty$
and~$f\in C^1(\R)$. Assume that
\begin{equation}\label{mono}
{\mbox{$\partial_{x_2} u(x)>0$
for any $x\in\R^2$.}}\end{equation}
Then, $u$ is necessarily one-dimensional.
\end{thm}

The proof of Theorem~\ref{DG-2} relies on a technique
introduced by~\cite{BCN} and refined in~\cite{AC},
which reduced the symmetry property to a Liouville type property
for an associated equation (of course, differently from
the classical case, we will have to deal with equations,
and in fact inequalities, of integral type, in which the
appropriate simplifications are more involved).

For the existence of one-dimensional solutions of~\eqref{EQ}
under quite general conditions, see Theorem~3.1(b)
in~\cite{Bates}.

The rest of the paper is devoted to the proof of Theorem \ref{DG-2}.

\begin{proof}[Proof of Theorem \ref{DG-2}]
We observe that
\begin{equation}\label{R1}
\begin{split}
&2\int_{\R^2} {\mathcal{L}} f(x)\,g(x)\,dx
= 
2\int_{\R^2}\left[\int_{\R^2}
\big( f(x)-f(y)\big)\,k(x-y)\,dy\right]\,g(x)\,dx
\\ &\qquad=
\int_{\R^2}\left[\int_{\R^2}
\big( f(x)-f(y)\big)\,k(x-y)\,dy\right]\,g(x)\,dx
\\ &\qquad\qquad+
\int_{\R^2}\left[\int_{\R^2}
\big( f(y)-f(x)\big)\,k(x-y)\,dx\right]\,g(y)\,dy
\\ &\qquad=
\int_{\R^2}\int_{\R^2}
\big( f(x)-f(y)\big)\,
\big( g(x)-g(y)\big)\,
k(x-y)\,dx\,dy.
\end{split}
\end{equation}
Now we let~$u_i:=\partial_{x_i} u$, for~$i\in\{1,2\}$.
In light of~\eqref{mono}, we can define
\begin{equation}\label{DF:v}
v(x):=\frac{u_1(x)}{u_2(x)}.\end{equation}
Also, fixed~$R>1$ (to be taken as large as we wish in the sequel),
we consider a cut-off function~$\tau:=\tau_R\in C^\infty_0(B_{2R})$,
such that~$\tau=1$ in~$B_R$ and~$|\nabla \tau|\le CR^{-1}$,
for some~$C>0$.

By~\eqref{EQ}, we have that
\begin{equation}\label{5bis}\begin{split}
& f'\big(u(x)\big)\,u_i(x)=
\partial_{x_i} \left( f\big(u(x)\big)\right)\\
&\qquad=
\partial_{x_i} \big({\mathcal{L}} u(x)\big)
=\partial_{x_i} \left( \int_{\R^2}
\big( u(x)-u(x-\zeta)\big)\,k(\zeta)\,d\zeta \right)
\\ &\qquad= \int_{\R^2}
\big( u_i(x)-u_i(x-\zeta)\big)\,k(\zeta)\,d\zeta 
= \int_{\R^2}
\big( u_i(x)-u_i(y)\big)\,k(x-y)\,dy\\
&\qquad= {\mathcal{L}} u_i(x).
\end{split}\end{equation}
Accordingly,
\begin{eqnarray*}
&& f'(u)\,u_1 u_2 = {\mathcal{L}} u_1 \,u_2\\
{\mbox{and }} && f'(u)\,u_1 u_2 = {\mathcal{L}} u_2 \,u_1.
\end{eqnarray*}
By subtracting these two identities
and using~\eqref{DF:v}, we obtain
$$ 0={\mathcal{L}} u_1 \,u_2-
{\mathcal{L}} u_2 \,u_1 = {\mathcal{L}} (v u_2) \,u_2
-{\mathcal{L}} u_2 \,(v u_2).$$
Now, we multiply by~$2\tau^2 v$ and we integrate.
Hence, recalling~\eqref{R1}, we conclude that
\begin{equation}\label{XR-1}
\begin{split}
0\, &= 2\int_{\R^2} {\mathcal{L}} (v u_2)(x) \,(\tau^2 v u_2)(x)\,dx
- 2\int_{\R^2}
{\mathcal{L}} u_2(x) \,(\tau^2 v^2 u_2)(x)\,dx\\
&=
\int_{\R^2}\int_{\R^2}
\big( vu_2(x)-vu_2(y)\big)\,
\big( \tau^2 v u_2(x)-\tau^2 v u_2(y)\big)\,
k(x-y)\,dx\,dy
\\ &\qquad-
\int_{\R^2}\int_{\R^2}
\big( u_2(x)-u_2(y)\big)\,
\big( \tau^2 v^2 u_2(x)-\tau^2 v^2 u_2(y)\big)\,
k(x-y)\,dx\,dy\\
&=: I_1-I_2.
\end{split}
\end{equation}
By writing
$$ vu_2(x)-vu_2(y) =
\big(u_2(x)-u_2(y)\big)\,v(x) + \big(v(x)-v(y)\big)\, u_2(y),$$
we see that
\begin{equation}\label{XR-2}\begin{split}
I_1\,& =
\int_{\R^2}\int_{\R^2}
\big(u_2(x)-u_2(y)\big)\,\big( \tau^2 v u_2(x)-\tau^2 v u_2(y)\big)\,
v(x)\,k(x-y)\,dx\,dy
\\&\qquad+
\int_{\R^2}\int_{\R^2} \big(v(x)-v(y)\big)\, 
\big( \tau^2 v u_2(x)-\tau^2 v u_2(y)\big)\,
u_2(y)\,k(x-y)\,dx\,dy
.\end{split}\end{equation}
In the same way, if we write
$$ \tau^2 v^2 u_2(x)-\tau^2 v^2 u_2(y)
=\big(
\tau^2 v u_2(x)-
\tau^2 v u_2(y)\big)\,v(x)
+\big(v(x)-v(y)\big)
\,\tau^2 v u_2(y),$$
we get that
\begin{equation}\label{XR-3}\begin{split}
I_2\,&=\int_{\R^2}\int_{\R^2}
\big( u_2(x)-u_2(y)\big)\,
\big(
\tau^2 v u_2(x)-
\tau^2 v u_2(y)\big)\,v(x)
\,k(x-y)\,dx\,dy\\&\qquad+
\int_{\R^2}\int_{\R^2}
\big( u_2(x)-u_2(y)\big)\,
\big(v(x)-v(y)\big)\,\tau^2 v u_2(y)\,k(x-y)\,dx\,dy
.\end{split}\end{equation}
By~\eqref{XR-2} and~\eqref{XR-3}, after a simplification we obtain that
\begin{eqnarray*}
I_1-I_2&=&
\int_{\R^2}\int_{\R^2} \big(v(x)-v(y)\big)\,
\big( \tau^2 v u_2(x)-\tau^2 v u_2(y)\big)\,
u_2(y)\,k(x-y)\,dx\,dy
\\ &&\qquad- \int_{\R^2}\int_{\R^2}
\big( u_2(x)-u_2(y)\big)\,
\big(v(x)-v(y)\big)\,\tau^2 v u_2(y)\,k(x-y)\,dx\,dy.\end{eqnarray*}
Now we notice that
\begin{eqnarray*}
&& \tau^2 v u_2(x)-\tau^2 v u_2(y)\\&=&
\big(v(x)-v(y)\big)\,\tau^2(x) \,u_2(x)
+\big(\tau^2 (x)- \tau^2(y) \big)\,u_2(x)\,v(y)
+\big( u_2(x)-  u_2(y)\big)\,\tau^2(y)\,v(y)
,
\end{eqnarray*}
and so
\begin{eqnarray*}
I_1-I_2&=&
\int_{\R^2}\int_{\R^2} \big(v(x)-v(y)\big)^2\,\tau^2(x) \,u_2(x)\,
u_2(y)\,k(x-y)\,dx\,dy \\ &&\quad+
\int_{\R^2}\int_{\R^2} \big(v(x)-v(y)\big)\,
\big(\tau^2 (x)-
\tau^2 (y)\big)\,v(y)
\,u_2(x)\,u_2(y)\,k(x-y)\,dx\,dy
.\end{eqnarray*}
Thus, using this and~\eqref{XR-1},
and recalling~\eqref{ASS-kernel} and the support properties
of~$\tau$, we deduce that
\begin{equation}\label{J24}
\begin{split}
J_1\,&:=\int_{\R^2}\int_{\R^2} \big(v(x)-v(y)\big)^2\,\tau^2(x) \,u_2(x)\,
u_2(y)\,k(x-y)\,dx\,dy \\ &\le
\iint_{{\mathcal{R}}_R} \big|v(x)-v(y)\big|\,
\big|\tau(x)-\tau (y)\big|\,
\big|\tau(x)+\tau (y)\big|\,
|v(y)|\,u_2(x)\,u_2(y)\,k(x-y)\,dx\,dy\\ &=:J_2
,\end{split}
\end{equation}
where
\begin{eqnarray*}
{{\mathcal{R}}_R} &:=& \{(x,y)\in\R^2\times\R^2 {\mbox{ s.t. }}
|x-y|\le R_0 \}\cap {{\mathcal{S}}_R}
\\ {\mbox{and }}\;
{\mathcal{S}}_R &:=& \Big( (B_{2R}\times B_{2R})\setminus (B_{R}\times B_{R}) \Big)
\;\cup \;\Big(
B_{2R}\times (\R^2\setminus B_{2R})\Big)\; \cup\; \Big(
(\R^2\setminus B_{2R})\times B_{2R}\Big).
\end{eqnarray*}
We use the symmetry in the~$(x,y)$ variables and the substitution~$\zeta:=
x-y$ to see that
\begin{equation}\label{measure}
\begin{split}
|{\mathcal{R}}_R|\, &\le \big| \{ |x-y|\le R_0 \}\cap \{|x|,|y|\le 2R\} \big|
+ 2\, \big| \{ |x-y|\le R_0 \}\cap \{|x|\le 2R\le|y|\} \big|
\\ &\le 3\, \int_{B_{R_0}} \left[ \int_{B_{2R}} \, dx\right]\,d\zeta
\\ &\le C R^2,\end{split}
\end{equation}
for some~$C>0$, possibly depending on~$R_0$.

Moreover, making use of the H\"older Inequality, we see that
\begin{equation}\label{J22}
\begin{split}
& J_2^2 \le 
\iint_{{\mathcal{R}}_R} \big(v(x)-v(y)\big)^2\,
\big(\tau(x)+\tau (y)\big)^2\,
\,u_2(x)\,u_2(y)\,k(x-y)\,dx\,dy
\\&\qquad\qquad\cdot
\iint_{{\mathcal{R}}_R} \big(\tau(x)-\tau (y)\big)^2\,
v^2(y)\,u_2(x)\,u_2(y)\,k(x-y)\,dx\,dy
.\end{split}\end{equation}
Now we claim that
\begin{equation}\label{HA}
u_2(x)\le C\,u_2(y)
\end{equation}
for any~$(x,y)\in{{\mathcal{R}}_R}$, for a suitable~$C>0$, possibly depending on~$R_0$.
For this, fix~$x$ and let~$\Omega:=B_{R_0}(x)$.
Then we use the
Harnack Inequality for integral equations
(recall~\eqref{ASS-kernel}, \eqref{mono}
and~\eqref{5bis}, and see  
Corollary~1.7 in~\cite{Coville}), to obtain that
$$ u_2(x)\le \sup_\Omega u_2 \le C\,\inf_\Omega u_2\le C\,u_2(y),$$
which establishes \eqref{HA}.

{F}rom~\eqref{DF:v} and~\eqref{HA},
we obtain that
$$ \big(\tau(x)-\tau (y)\big)^2\,
v^2(y)\,u_2(x)\,u_2(y) \le CR^{-2} \,v^2(y)\,u_2^2(y)= CR^{-2}\,u_1^2(y)
\le CR^{-2},$$
for some~$C>0$.
Hence, by \eqref{measure},
$$ \iint_{{\mathcal{R}}_R} \big(\tau(x)-\tau (y)\big)^2\,
v^2(y)\,u_2(x)\,u_2(y)\,k(x-y)\,dx\,dy\le C,$$
for some~$C>0$.
Therefore, recalling~\eqref{J22},
\begin{equation}\label{J23}
J_2^2 \le C \iint_{{\mathcal{R}}_R} \big(v(x)-v(y)\big)^2\,
\big(\tau(x)+\tau (y)\big)^2\,
\,u_2(x)\,u_2(y)\,k(x-y)\,dx\,dy.\end{equation}
Hence, since
$$ \big(\tau(x)+\tau (y)\big)^2=
\tau^2(x)+\tau^2 (y)+2\tau(x)\,\tau(y)\le
3\tau^2(x)+3\tau^2 (y),$$
we can use the symmetric role played by~$x$ and~$y$ in~\eqref{J23}
and obtain that
$$ J_2^2 \le C \iint_{{\mathcal{R}}_R} \big(v(x)-v(y)\big)^2\,
\tau^2(x)\,
\,u_2(x)\,u_2(y)\,k(x-y)\,dx\,dy,$$
up to renaming~$C>0$.
So, we insert this information into~\eqref{J24}
and we conclude that
\begin{equation}\label{J25}
\begin{split}
& \left[ \iint_{\R^2\times\R^2} \big(v(x)-v(y)\big)^2\,\tau^2(x) \,u_2(x)\,
u_2(y)\,k(x-y)\,dx\,dy \right]^2=J_1^2\\ &\quad\le J_2^2\le 
C \iint_{{\mathcal{R}}_R} \big(v(x)-v(y)\big)^2\,
\tau^2(x)\,
\,u_2(x)\,u_2(y)\,k(x-y)\,dx\,dy.
,\end{split}\end{equation}
for some~$C>0$.

Since clearly~${\mathcal{R}}_R\subseteq\R^2\times\R^2$,
we can simplify the estimate in~\eqref{J25} by writing
$$ \iint_{\R^2\times\R^2} \big(v(x)-v(y)\big)^2\,\tau^2(x) \,u_2(x)\,
u_2(y)\,k(x-y)\,dx\,dy \le C.$$
In particular, since~$\tau=1$ in~$B_R$,
$$ \iint_{B_R \times B_R} \big(v(x)-v(y)\big)^2\,u_2(x)\,
u_2(y)\,k(x-y)\,dx\,dy \le C.$$
Since~$C$ is independent of~$R$, we can send~$R\to+\infty$
in this estimate and obtain that the map
$$ \R^2\times\R^2\ni (x,y)\mapsto
\big(v(x)-v(y)\big)^2\,u_2(x)\,
u_2(y)\,k(x-y)$$
belongs to~$L^1(\R^2\times\R^2)$.

Using this and the fact that~${\mathcal{R}}_R$ approaches the
empty set as~$R\to+\infty$, we conclude that
$$ \lim_{R\to+\infty}
\iint_{{\mathcal{R}}_R} \big(v(x)-v(y)\big)^2 \,u_2(x)\,
u_2(y)\,k(x-y)\,dx\,dy =0.$$
Therefore, going back to~\eqref{J25},
\begin{eqnarray*}
&& 
\left[ \iint_{\R^2\times\R^2} \big(v(x)-v(y)\big)^2 \,u_2(x)\,
u_2(y)\,k(x-y)\,dx\,dy \right]^2\\
&\le& \lim_{R\to+\infty}
\left[ \iint_{\R^2\times\R^2} \big(v(x)-v(y)\big)^2\,\tau^2(x) \,u_2(x)\,
u_2(y)\,k(x-y)\,dx\,dy \right]^2\\ &\le&
\lim_{R\to+\infty}
C \iint_{{\mathcal{R}}_R} \big(v(x)-v(y)\big)^2\,
\tau^2(x)\,
\,u_2(x)\,u_2(y)\,k(x-y)\,dx\,dy.
\\ &=&0.\end{eqnarray*}
This and~\eqref{mono} imply that~$\big(v(x)-v(y)\big)^2 \,k(x-y)=0$
for any~$(x,y)\in\R^2\times\R^2$.
Hence, recalling~\eqref{ASS-kernel}, we have that~$v(x)=v(y)$
for any~$x\in\R^2$ and any~$y\in B_{r_0}(x)$.

As a consequence, the set~$\{y\in\R^2 {\mbox{ s.t. }} v(y)=v(0)\}$
is open and closed in~$\R^2$, and so, by connectedness,
we obtain that~$v$ is constant, say~$v(x)=a$ for some~$a\in\R$.
So we define~$\omega:=\frac{(a,1)}{\sqrt{a^2+1}}$ and we observe that 
$$ \nabla u(x) = u_2(x)\,(v(x),1)=u_2(x)\,{\sqrt{a^2+1}}\;\omega.$$ 
Thus, if~$\omega\cdot y=0$
then
$$ u(x+y)-u(x)=\int_0^1 \nabla u(x+ty)\cdot y\,dt
=\int_0^1 u_2(x+ty)\,{\sqrt{a^2+1}}\;\omega\cdot y\,dt=0.$$
Therefore, if we
set~$u_\star(t):=u(t\omega)$
for any~$t\in\R$, and we write any~$x\in\R^2$ as
$$ x=\left({\omega}\cdot x\right) \omega+y_x$$
with~$\omega\cdot y_x=0$, we conclude that
$$ u(x)=u\left(
\left({\omega}\cdot x\right) \omega +y_x\right)
= u\left(
\left({\omega}\cdot x\right) \omega \right)
=u_\star\left({\omega}\cdot x\right).$$
This completes the proof of Theorem~\ref{DG-2}.\end{proof}

For completeness, we observe that a more general
version of Theorem~\ref{DG-2} holds true,
namely if we replace assumption~\eqref{mono}
with a ``stability assumption'' in the sense of~\cite{AC}: the precise statement
goes as follows:

\begin{thm}\label{DG-2-STAB}
Let~$u$ be a solution of~\eqref{EQ}
in the whole of~$\R^2$, with~$\|u\|_{C^1(\R^2)}<+\infty$
and~$f\in C^1(\R)$. Assume that there exists~$\psi>0$ which solves
\begin{equation*}
{\mbox{${\mathcal{L}}\psi(x)=f'\big(u(x)\big)\,\psi(x)$
for any $x\in\R^2$.}}\end{equation*}
Then, $u$ is necessarily one-dimensional.
\end{thm}

Notice that, in this setting, Theorem~\ref{DG-2}
is a particular case of Theorem~\ref{DG-2-STAB},
choosing~$\psi:=\partial u_2$ and recalling~\eqref{5bis}.

The proof of Theorem~\ref{DG-2-STAB} is like the one
of Theorem~\ref{DG-2}, with only a technical difference:
instead of~\eqref{DF:v}, one has to define, for~$i\in\{1,2\}$,
$$ v(x):=\frac{u_i(x)}{\psi(x)}.$$
Then the proof of Theorem~\ref{DG-2} goes through
(replacing~$u_2$ with~$\psi$ when necessary)
and implies that~$v$ is constant, i.e.~$u_i =a_i\psi$,
for some~$a_i\in\R$. This gives that~$\nabla u(x)=\psi(x)\,(a_1,a_2)$,
which in turn implies the one-dimensional symmetry of~$u$.
\medskip

Also, we think that it is an interesting open problem to investigate
if symmetry results in the spirit of Theorems~\ref{DG-2}
and~\ref{DG-2-STAB} hold true in higher dimension.

\end{document}